\documentclass[12pt,english]{article}

\usepackage{graphicx}
\usepackage{amsfonts}
\usepackage{amssymb}
\usepackage{amsthm}
\usepackage{amsmath}

\usepackage[T1]{fontenc}
\usepackage{babel}
\usepackage[utf8]{inputenc}
\usepackage{authblk}

\newtheorem{thm}{Theorem}
\newtheorem{prop}{Proposition}[section]
\newtheorem{lem}{Lemma}[section]
\newtheorem{cor}{Corollary}[section]
\newtheorem{fact}{Fact}[section]
\newtheorem{con}{Conjecture}
\newtheorem{que}{Question}
\newtheorem{rmk}{Remark}[section]

\newtheorem{thm_proof}{Theorem}

\theoremstyle{definition}
\newtheorem{exm}{Example}
\newtheorem{dfn}{Definition}

\def \tp {{\rm tp}}

\def\cl{{\rm cl}}

\def \cl {{\rm cl}}

\def \dom {{\rm dom}}

\title{Around Podewski's conjecture}

\author{Krzysztof Krupi\'nski\thanks{Research supported by the Polish Government grant N N201
545938.},
 Predrag Tanovi\'c\thanks{Supported by the Ministry of Education and
Science of Serbia},
 Frank O. Wagner\thanks{Partially supported by the Agence Nationale
de la Recherche under project ANR-09-BLAN-0047.}}

\begin{document}
\maketitle

\begin{abstract}
A long-standing conjecture of Podewski states that every minimal
field is algebraically closed. Known in positive characteristic,
it remains wide open in characteristic zero. We reduce Podewski's
conjecture to the (partially) ordered case, and we conjecture that
such fields do not exist. We prove the conjecture in case the
incomparability relation is transitive (the almost linear case).

We also study minimal groups with a (partial) order, and give a
complete classification of almost linear minimal groups as certain
valued groups.\end{abstract}

%\section{Introduction}

Recall that an infinite first-order structure is {\em
minimal} if every definable (with parameters) subset is either
finite or co-finite (of finite complement). Minimal pure
groups were classified by Reineke \cite{R}; they are either
abelian divisible with only finitely many elements of any given
finite order, or elementary abelian of prime exponent. As for
minimal fields, it is well-known that every algebraically closed
pure field is minimal; the converse was predicted by Podewski
\cite{Po} forty years ago:

\begin{con}\label{C1}
A minimal field is algebraically closed.
\end{con}
It was shown by the third author in positive characteristic \cite{W}, but the
characteristic zero case remains one of the oldest unsolved problems in algebraic model theory.

We shall say that a structure is {\em ordered} if it has a definable
strict partial order on singletons which is not definable from equality.
Otherwise the structure is {\em unordered}. Please note that in this paper ordered fields or groups are ordered structures in the above sense rather than in the usual algebraic sense.
\begin{lem} A minimal ordered structure $M$ has an infinite chain.\end{lem}
\begin{proof} If there is a chain of order type $\omega^*$, the reverse order of $\omega$, we are done. So assume that there is no such chain; we shall construct inductively a chain of order type $\omega$.

Suppose the set $X$ of minimal elements is co-finite. Then the elements of $X$ are incomparable. By minimality, for each $y\in M\setminus X$ the sets of elements less than $y$, bigger than $y$ and incomparable to $y$ are definable from equality. Hence the partial order is definable from equality, a contradiction.

Therefore the set of minimal elements is finite. If it is empty, there is a chain of order type $\omega^*$, a contradiction. Hence for some minimal element $x_0$ the set $X_0=\{y:x_0<y\}$ must be infinite, and hence co-finite by minimality. But if we can define the order on some co-finite set from equality, then, by minimality, we can define the order on $M$ from equality. By induction, we obtain an infinite increasing chain.\end{proof}

The minimal total orders are just $(\omega+n,<)$,
$((\omega+n)^*,<)$ and $(\omega+\omega^*,<)$. Most of
the known non-linear, ordered, minimal structures derive from them
by replacing elements by sufficiently large finite antichains and
then adding a finite set arbitrarily. For example, consider
$\{(n,m)\in\omega\times\omega \,|\,m\leq n\}$ and order it by
$(n,m)<(n',m')$ iff $n<n'$. Here, the order is not far from being
linear, namely the incomparability relation (defined by
$x\sim y$ iff $\neg(x<y\vee y<x)$) is transitive and hence an equivalence relation; after factoring it out we end up with $(\omega,<)$. The
minimal structures in which there is a definable order with an infinite chain
such that incomparability is transitive will be called {\em almost
linear} (see Definition \ref{D1} and Remark \ref{Rem1}).

In Section \ref{Section 3}, we prove Conjecture \ref{C1} for unordered fields:
\begin{thm}\label{T2} A minimal unordered field is algebraically closed.
\end{thm}
Thus Podewski's conjecture is reduced to the ordered case:
\begin{con}\label{C2}
There is no minimal ordered field of characteristic zero.
\end{con}
We study minimal ordered groups in Section \ref{Section 4} and show:
\begin{thm}\label{T3}
An almost linear minimal group $G$ is either elementary abelian of exponent $p$ or a finite sum of Pr\"ufer $p$-groups for a fixed prime $p$. In particular, it is a torsion group.
\end{thm}
This implies immediately:
\begin{thm}\label{T1} There is no almost linear minimal field.
\end{thm}

It is thus natural to ask:
\begin{que}\label{Question 1}
Is every minimal ordered group a torsion group?
\end{que}
Theorem \ref{T1} implies that a possible counterexample to
Conjecture \ref{C2} would have to be a field which is not almost
linear. It is hard to believe that such structures exist. In particular, all
known examples of minimal ordered structures are almost linear.
\begin{que}\label{Question 2}
Does there exist a minimal ordered structure [group] which is not almost linear?
\end{que}
Note that the analogue of Theorem \ref{T1} for quasi-minimal fields (uncountable fields whose definable subsets are countable or co-countable) is false: There is an almost linear quasi-minimal field \cite[Example 5.1]{PT}.

Finally, in Section \ref{Section 5} we classify almost
linear minimal groups, showing in particular that all cases in the conclusion of Theorem \ref{T3} can be realized (so the analogue of Theorem \ref{T1} for groups is false).

\section{Minimal structures with definable generic type}\label{Section 2}

Let $M$ be a minimal structure. The unique non-algebraic type $p \in S_1(M)$ will be called the {\em generic type} of $M$. In this section, we are interested in minimal structures
whose generic type is definable. This class is interesting because of the next lemma, noticed by A. Pillay.

\begin{lem}\label{L11}
The generic type $p$ of a minimal group $G$ is its unique generic type in the sense of left [and right] translates, i.e.\ a formula $\phi(x)$ (with parameters from $M$) belongs to $p$ if and only if finitely many left [right] translates of $\phi(G)$ cover $G$. In fact, $\phi(x) \in p$ if and only if two left [right] translates of $\phi(G)$ cover $G$. This characterization of $p$ implies that $p$ is definable over $\emptyset$.
\end{lem}

\begin{proof}
If finitely many translates of $\phi(G)$ cover $G$, then $\phi(G)$ is infinite, hence $\phi(x) \in p$. For the converse, suppose $\phi(x) \in p$. Then $G \setminus \phi(G)$ is a finite set $\{ g_1,\dots,g_n\}$. Since the sets $g_1\phi(G)^{-1}, \dots, g_n \phi(G)^{-1}$ are co-finite, their intersection is non-empty, so it contains an element $g$. It is clear now that $G =\phi(G) \cup g\phi(G)$.
\end{proof}

For the remainder of this section $M$ will be a minimal structure
whose generic type $p\in S_1(M)$ is definable over $\emptyset$. Let
$\bar{M}\succ M$ be a monster model. Then $p$ has a unique global heir $\bar p \in S_1(\bar{M})$, which is defined by the same defining scheme as $p$. For $C \subseteq \bar{M}$ a {\em generic Morley sequence} is a sequence $(a_i : i \in \kappa)$ such that $a_i \models\bar p|Ca_{<i}$ for all $i \in \kappa$.
Each generic Morley sequence over $C$ is indiscernible over $C$,
and the type over $C$ of such a sequence of a fixed length
$\kappa$ does not depend on its choice. Recall that for a formula
$\varphi(x;\overline{y})$ a $\varphi(x;\overline{y})$-definition of $p$ is denoted by $d_p x \varphi(x,\overline{y})$.

\begin{prop}\label{closure} For any $A\subseteq\bar M$ put
$$\cl(A)=\{x \in \bar M:\tp(x/A)\text{ is non-generic}\}.$$
Then $\cl$ is a closure operator on $\bar M$. In particular, it is idempotent.\end{prop}
\begin{proof} Clearly $\cl$ has finite character, and  $A\subseteq B$  implies $A\subseteq\cl(A)\subseteq \cl(B)$. Suppose $a\in\cl(\cl(A))$. Then there is a non-generic definable set $\phi(\bar M,\bar b,A)$ containing $a$,
and for every $b_i\in\bar b$ a non-generic definable set
$\phi_i(\bar M,A)$ containing $b_i$. If $a$ were generic over $A$, then
$A$ would satisfy
$$d_px\exists\bar y\,\{\phi(x,\bar y,X)\land\neg d_pz\phi(z,\bar y,X)\land\bigwedge_i[\phi_i(y_i,X)\land\neg d_pz\phi_i(z,X)]\}.$$
Hence, there is $A_0\subseteq M$ satisfying this formula. But in $M$ a non-generic formula defines a finite set, so the formula
$$\exists\bar y\,\{\phi(x,\bar y,A_0)\land\neg d_pz\phi(z,\bar y,A_0)\land\bigwedge_i[\phi_i(y_i,A_0)\land\neg d_pz\phi_i(z,A_0)]\},$$
defines a finite set, which cannot be generic, a contradiction. Thus $\cl(\cl(A))=\cl(A)$.\end{proof}

We shall now prove a version of the dichotomy theorem for minimal
structures from \cite{T2} in our context. The proof uses arguments
from \cite{PT}, where a similar result was proved for locally
strongly regular types.

\begin{prop}\label{P1} Let $(a,b)$ be a generic Morley sequence over $M$. Then exactly one of the following two cases holds.
\begin{itemize}
\item{\bf Symmetric:} $\tp(a,b/M)=\tp(b,a/M)$.\\
In this case, a generic Morley sequence over any $C\subseteq \bar{M}$ containing $A$ is totally indiscernible over $C$ (i.e.\ indiscernible over $C$ as a set).
\item{\bf Asymmetric:} $\tp(a,b/M)\neq\tp(b,a/M)$.\\
In this case, there is an $M$-definable strict partial order on $M$ such that
$M<a<b$. Moreover, $(M,\leq)$ is a directed, well partial order
having infinite increasing chains and no such chain of order type
$\omega+1$.\end{itemize}\end{prop}

\begin{proof}  Suppose that $\tp(a,b/M)=\tp(b,a/M)$. This implies that for any
$\phi(x,y,\overline{z})$ without parameters and for any $\overline{m}\in M$ we have
$$\models d_px\,d_py(\phi(x,y,\overline{m})\leftrightarrow\phi(y,x,\overline{m})).$$
Then the same formula is satisfied by any $\overline{c}\in C$ in place of $\overline{m}$, and so for every generic Morley sequence $(a',b')$ over $C$ we have $\tp(a',b'/C)=\tp(b',a'/C)$.  By induction, it follows easily that generic Morley sequences of any length are totally indiscernible.

Now, suppose $\tp(a,b/M)\neq\tp(b,a/M)$. We will find
$\phi(x,y)\in\tp(a,b/M)$ such that:
\begin{enumerate}
\item $\models \forall x,y( \phi(x,y)\rightarrow \neg
\phi(y,x))$ \ \ ($\phi(x,y)$ is asymmetric);
\item $\phi(M,b)=M$ \ (i.e. $\phi(c,y)\in p(y)$ for all $c\in
M$);
\item $\phi(a,y)\vdash p(y)$.
\end{enumerate}
Since $p$ is definable, it has a unique heir
and a unique coheir in $S_1(Ma)$; since $\tp(a,b/M)\neq\tp(b,a/M)$,
the two must be distinct. Therefore $\tp(b/Ma)$ is not a coheir, so
there is $\phi(a,y)\in\tp(b/Ma)$ which is satisfied by no element
of $M$. Since $\tp(a/Mb)$ is a coheir, $\phi(M,b)$ is infinite.
Definability of $p$ implies that $\phi(M,b)$ is co-finite;
modifying it slightly on $M$, we may assume $\phi(M,b)=M$.
Moreover, $\phi(b,x)\notin\tp(a/Mb)$ as $\tp(a/Mb)$ is a coheir,
so $(a,b)$ satisfies $\phi'(x,y)=\phi(x,y)\land\neg\phi(y,x)$.
Clearly $\phi'$ satisfies conditions (1)--(3).

We leave to the reader to verify  that  the formula
$$\phi(x,y)\wedge \forall t(\phi(y,t)\rightarrow \phi(x,t))$$
defines a strict partial order; denote it by $x<y$. Now, we prove
that for any $c\in M$ we have $c<b$, i.e.
$$\models \phi(c,b)\wedge \forall
t(\phi(b,t)\rightarrow \phi(c,t)).$$ Condition (2) implies
$\models \phi(c,b)$. Let $d$ be such that $\models\phi(b,d)$.
Then, by (3), $d$ realizes $p$, so $\phi(c,y)\in p(y)$ implies
$\models\phi(c,d)$. Thus, $c<b$. Since this holds for all $c\in M$
and $\tp(a/bM)$ is finitely satisfiable in $M$, we conclude that
$a<b$. As $\tp(a/M)=\tp(b/M)$ and $M<b$, we also get $M<a$.

Now, the formula $x<b$ belonging to $\tp(a/bM)$ is satisfied in $M$, say by $c_1\in M$. Then
$c_1<b$; in fact, since $M<a$, we also have $c_1<a$. Hence,
the formula $c_1<x<b$ belongs to $\tp(a/bM)$, and so there is $c_2\in M$ such that
$c_1<c_2<b$; then also $c_2<a$, so the formula $c_1<c_2<x<b$ belongs to $\tp(a/bM)$. Continuing in this way, we get an infinite increasing chain in $M$.  We leave to
the reader to verify that $\leq$ is a directed, well partial order with no chains of order type $\omega+1$.
\end{proof}

The above proposition leads to the following definition.
\begin{dfn}\label{Def_symmetric}
$M$ is {\em symmetric} if $\tp(a,b/M)=\tp(b,a/M)$ for each/some generic Morley sequence $(a,b)$ over $M$; equivalently, if generic Morley sequences (of arbitrary length) over any set $C$
are totally indiscernible over $C$. Otherwise $M$ is {\em asymmetric}.
\end{dfn}

From now on, whenever a definable partial order $<$ is clear from the context, $x \sim y$ will be defined as $\neg(x<y\vee y<x)$.
\begin{dfn}\label{D1}  A definable partial order $<$ on $M$ with infinite chains is {\em almost linear} if $\sim$ is an equivalence relation on $p(\bar M)$. We call $M$ {\em almost linear} if such an order exists.
\end{dfn}

\begin{rmk}\label{Rem1} Suppose that $<$ is a definable partial order with infinite chains on some minimal structure. Then the following conditions are equivalent.\begin{enumerate}
\item $<$ is almost linear.
\item After a modification of $<$ on a finite set, incomparability $\sim$ becomes an equivalence relation (we allow here modifications of $<$ between elements of this finite set and all other elements, but in such a
way that the resulting order is definable).
\item After a modification of $<$ on a finite set, incomparability $\sim$ becomes an equivalence relation, and the set of equivalence classes has order type $(\omega,<)$, $(\omega^*,<)$ or $(\omega + \omega^*,<)$.\end{enumerate}
\end{rmk}
\begin{proof}
The implications $(3) \rightarrow (2)  \rightarrow (1)$ are
clear. In order to see that $(1) \rightarrow (2)$, one should use
compactness and the minimality of $M$. Finally, for $(2) \rightarrow (3)$ it suffices to notice that the order type of a minimal linear order is either one of the types listed in $(iii)$, or
$(\omega+n,<)$ or $((\omega+n)^*,<)$.\end{proof}

\begin{rmk}\label{Remadditional}
Assume that $M$ is ordered by $<$ with an infinite increasing chain. Then  $C:=\{c\in M\,|\,c<x\in p(x)\}$ is a definable, co-finite subset of $M$, and hence $<$ is a well partial order and $M$ is asymmetric. After modifying $<$ so that the elements of $M\setminus C$ are below all the others, we get that $C=M$, and then $<$ is a directed, well partial order having infinite
increasing chains and no such chain of order type $\omega +1$. If
in addition $<$ is almost linear, then after modifying $<$ on a  finite set we get that $M/{\sim}$ is ordered of order type $\omega$.
\end{rmk}

Of course, analogous observations are true when $M$ contains an infinite decreasing chain. Proposition \ref{P1} together with Remark \ref{Remadditional} yield the following observation.

\begin{rmk}\label{Remnon-sym}
Let $M$ be a minimal structure whose generic type is definable.
Then $M$ is asymmetric iff $M$ is ordered.\end{rmk}

We finish this section with an  example of an ordered minimal
structure which is due to Grzegorz Jagiella. Although the
structure is almost linear, it has a definable order with infinite
chains which is not almost linear.

\begin{exm}\label{EJagiella}
Let $M=\omega\times \{l,r\}$. Define an order $<$ on $M$ by
putting the natural orders on $\omega\times\{l\}$ and on
$\omega\times\{r\}$, together with
$$(x,l)<(y,r) \iff x+2 \leq y \;\;\;\; \mbox{and} \;\;\;\; (y,r)<(x,l) \iff y+2 \leq x$$
for all natural numbers $x$ and $y$.

We leave to the reader to verify that $(M,<)$ is minimal. For
$n\in\omega$ define $a_{2n}=(2n,l)$ and $a_{2n+1}=(2n+1,r)$. Then
$a_{n}$ and $a_{n+1}$ are incomparable for all $n$, and so $<$ is
not almost linear.

Now, we show that $M$ interprets $(\omega,<)$. First, note that
`$y$ is maximal incomparable to $x$' is a definable function
$f(x)=y$. Then for  $x,y\in M$ define:
$$x<'y \ \ \ \mbox{if and only if} \ \ \ x<y \ \ \mbox{or} \ \   y=f(x).$$
Then $<'$-incomparability  is an equivalence relation with 2-element
classes $\{(n,r),(n,l)\}$, and, after factoring it out, we end up
with $(\omega,<)$. By Remark \ref{Rem1}, $<'$ is almost linear in
$M$, so $M$ is almost linear.
\end{exm}

\section{The symmetric case}\label{Section 3}

In this section, we shall prove that symmetric minimal fields are
algebraically closed. The following is a version of the
corresponding result for locally strongly regular types from
\cite{PT}. It is adapted to the context of minimal structures.

\begin{prop}\label{exchange} Let $M$ be a symmetric minimal structure with definable (over $\emptyset$) generic type, $\bar M$ its monster model and $\cl$ the closure operator from Proposition \ref{closure}. Then $(\bar M, \cl)$ is an infinite dimensional  pre-geometry, and $a_1,\dots, a_n$ is $\cl$-independent over $A$ if and only if it is a generic Morley sequence over $A$. In particular, if $(a_1,\ldots,a_n)$ is a generic Morley sequence over $A$ and $a_1,\ldots,a_n\in\cl(A,b_1,\ldots,b_n)$, then $(b_1,\ldots,b_n)$ is also a generic Morley sequence over $A$.\end{prop}
\begin{proof} We only have to prove the exchange property. So consider $a\in\cl(Ab)\setminus\cl(A)$. Then $a\notin\cl(A)$, so $\tp(a/A)$ is generic. Suppose $b\notin\cl(Aa)$. Then $\tp(b/Aa)$ is generic, and $(a,b)$ is a generic Morley sequence over $A$. By symmetry, $(b,a)$ is a generic Morley sequence over $A$. In particular, $a\notin\cl(Ab)$, a contradiction. Thus, $\cl$ satisfies the exchange property.
\end{proof}

It follows that the generic type $p$ is generically stable, and orthogonal to all non-generic types (see \cite{PT} for the definitions; we will not use this terminology in this paper).

Now, it is straightforward to deduce Theorem \ref{T2} from \cite[Theorem 1.13]{HLS}, where it is proved that any field carrying a
pre-geometry with certain homogeneity properties is algebraically
closed; it is based on Macintyre's proof that $\omega_1$-categorical fields are algebraically closed \cite{mac}. Here, we will give an alternative proof based on an argument of Wheeler \cite[Theorem 2.1]{Wh}; it was also used by Pillay
\cite[Proposition 5.2]{P} to prove that $\omega$-stable fields are
algebraically closed.

\begin{lem}\label{roots} If $K$ is a minimal field, then $(K^*)^n=K$ for any $n>0$. In particular $K$ is perfect.\end{lem}
\begin{proof} $(K^*)^n$ is an infinite subgroup of $K^*$, so, by minimality, it is co-finite and $(K^*)^n=K^*$.\end{proof}

\begin{thm_proof} A symmetric minimal field is algebraically closed.
\end{thm_proof}
\begin{proof} Let $K$ be a symmetric minimal field, $\bar K\succ K$ a monster model, and $p$ the generic type.
Let $F$ denote an algebraic closure of $\bar K$. Suppose for a
contradiction that some $\alpha_1\in F\setminus K$ is algebraic
over $K$. Let $f(x)=x^n+a_{n-1}x^{n-1}+\dots+a_0\in K[x]$ be the
minimal polynomial of $\alpha_1$ over $K$. Since $K$ is perfect,
$f$ has $n$ pairwise distinct roots $\alpha_1,\dots,\alpha_n$ in
$F$. Let $(t_0,\dots,t_{n-1})\in\bar K^n$ be a Morley sequence in
$p$ over $K$, and define $r_i=
t_0+t_1\alpha_i+\dots+t_{n-1}\alpha_i^{n-1}$ for $i=1,\dots,n$.
Then
$$\left(
\begin{array}{cccccccccc}
  1 & \alpha_1 & \dots  &\alpha_1^{n-1}   \\
 1 & \alpha_2 & \dots  &\alpha_2^{n-1}   \\
\vdots & \vdots & \dots  &\vdots   \\
1 & \alpha_n & \dots  &\alpha_n^{n-1}   \\
\end{array}\right)\left(
\begin{array}{cccccccccc}
  t_0   \\
t_1  \\
\vdots    \\
t_{n-1}   \\
\end{array}\right)=\left(
\begin{array}{cccccccccc}
  r_1   \\
r_2  \\
\vdots    \\
r_{n}   \\
\end{array}\right)$$
and, since the matrix is invertible, $(r_1,\dots,r_n)$ and
$(t_0,\dots,t_{n-1})$ are interalgebraic over
$K$.

Let $c_0,\dots,c_{n-1}$ be the symmetric functions of
$r_1,\dots,r_n$. Then the sequences $(c_0,\dots,c_{n-1})$ and
$(r_1,\dots,r_n)$ are interalgebraic over $K$, hence
$(t_0,\dots,t_{n-1})$ and $(c_0,\dots,c_{n-1})$ are interalgebraic
over $K$, too. Since $c_i,t_j$ are in $\bar K$, we can apply
Proposition \ref{exchange}  to conclude that $(c_0,\dots,c_{n-1})$
is a Morley sequence in $p$ over $K$. But for generic $x$ over
$c_1,\ldots,c_{n-1}$ the element
$$c'=-(x^{n}+c_{n-1}x^{n-1}+\dots+c_1x)$$
is again generic by Proposition \ref{exchange}, so it has the same type over $c_1,\ldots,c_{n-1}$ as $c_0$, and there is an automorphism $\sigma$ fixing $c_1,\ldots,c_{n-1}$ and moving $c'$ to $c_0$. Then $\sigma(x)$ is a zero of the polynomial
$$z^{n}+c_{n-1}z^{n-1}+\dots+c_1z+c_0.$$
Thus there is  $i$ such that
$$\sigma(x)=r_i=t_0+t_1\alpha_i+\dots+t_{n-1}\alpha_i^{n-1}\in \bar K.$$
This means that the degree of the minimal polynomial of $\alpha_i$
over $K(t_0,\dots,t_{n-1})$ is smaller than $n$. Since
$K(t_0,\dots,t_{n-1}) \subseteq {\bar K}$, this implies that the
degree of the minimal polynomial of $\alpha_i$ over ${\bar K}$ is
also smaller than $n$. On the other hand, as $K \prec {\bar K}$
and $f(x) \in K[x]$ is the minimal polynomial of $\alpha_i$ over
$K$ (so it is irreducible in $K[x]$), we get that $f(x)$ is also
irreducible in ${\bar K}[x]$, and so it is the minimal polynomial
of $\alpha_i$ over ${\bar K}$. This is a contradiction, because
$\deg(f)=n$.
\end{proof}

Notice the following consequence of Proposition \ref{P1} and Theorem \ref{T2}.
\begin{cor}
Each minimal field whose theory does not have the strict order property is algebraically closed. In particular, each minimal field whose theory is simple is algebraically closed.
\end{cor}

\section{Asymmetric minimal groups}\label{Section 4}

In this section, we shall show some general properties of
asymmetric minimal groups. In particular, an almost linear minimal group is either  elementary abelian or a finite sum of Pr\"ufer $p$-groups for some prime $p$; it follows that there is no almost linear minimal field.
Unfortunately, as far as asymmetric minimal groups in general are concerned, the following questions are still open; an affirmative answer would immediately imply Podewski's conjecture:

\begin{que}Is every asymmetric minimal group almost linear? Is it at least torsion?\end{que}

Given a minimal group $G$, a {\em generic element} is an element which is  generic over $G$, i.e.\ a realization (in a monster model $\bar G$) of the unique generic type $p \in S_1(G)$.
By Proposition \ref{P1} and Remark \ref{Remadditional}, whenever we are working in an asymmetric minimal group $(G,<,+,0,\dots)$, we can and do assume that $<$ is a directed, well, strict partial order with an infinite increasing chain and with no such chain of order type $\omega+1$, and such that $G<g$ for any generic $g$.

\begin{lem}\label{lincom}
Let $(G,<,+,0,\dots)$ be an asymmetric minimal group, $g$
generic, $g_1,\ldots,g_k<g$, and $n_1,\ldots,n_k$  integers. Then
$n_1g_1+\cdots+n_kg_k\not= g$.
\end{lem}
\begin{proof} Suppose $n_1g_1+\cdots+n_kg_k= g$. Then
\begin{equation}\label{eqn1}\exists x_1,\ldots, x_k\ \left(\bigwedge_{i=1}^kx_i<x\land\sum_{i=1}^kn_ix_i= x\right)\end{equation}
holds generically, and hence outside a finite set $X\subset G$. By
Reineke's result, $G$ is either elementary abelian, or abelian
divisible. In either case, any finitely generated subgroup is a proper
subgroup of $G$. Let $H$ be the subgroup generated by $X$.
Consider a minimal element $a$ in $G\setminus H$. As (\ref{eqn1}) holds outside $X\subseteq H$, it is satisfied by $a$. But any $a_1,\ldots, a_k<a$  are in $H$ by the minimality of $a$, as is $\sum_{i=1}^kn_ia_i$. So
$\sum_{i=1}^kn_ia_i\neq  a$, a contradiction.
\end{proof}

There is also a corresponding version for asymmetric minimal fields:
\begin{lem}\label{lincom_fields}
Let $(K,<,+,\cdot, 0,1,\dots)$ be an asymmetric minimal field,  $g$
generic, and $g_1,\ldots,g_k<g$. Let $f(x_1,\dots,x_k)$ be a rational function over $K$ such that the tuple $(g_1,\dots,g_k)$ is in its domain. Then $f(g_1,\dots,g_k)\not= g$.\end{lem}
\begin{proof}  Suppose $f(g_1,\dots,g_k)= g$. Then
\begin{equation}\label{eqn3}\exists \overline{x}=(x_1,\ldots, x_k)\ \left(\bigwedge_{i=1}^kx_i<x\land  \overline{x} \in \dom(f) \land f(\overline{x})= x\right)\end{equation}
holds generically, and hence outside a finite set $X\subset K$.
Let $F$ be the subfield of $K$ generated by $X$.
Since $K$ is minimal, it is closed under $n$-th roots for all $n>0$ by Lemma \ref{roots} and thus not finitely generated. Hence $F$ is a proper subfield; consider a minimal element $a$ in $K\setminus F$. As (\ref{eqn3}) holds outside $X\subseteq F$, it is satisfied by $a$. Consider any $a_1,\ldots, a_k<a$ such that $(a_1,\dots,a_k) \in \dom(f)$. Then $a_1,\dots,a_k \in F$ by
the minimality of $a$, hence $f(a_1,\dots,a_k) \in F$, and so
$f(a_1,\dots,a_k)\neq  a$, a contradiction.\end{proof}

Recall that a divisible abelian group $G$ splits as a direct sum of the torsion subgroup $Tor(G)$ and a direct sum of copies of ${\mathbb Q}$; furthermore, $Tor(G)$ is a direct sum of some numbers of copies of Pr\"ufer $p$-groups, where $p$ ranges over prime numbers.
Under the assumption that $G$ is not a finite sum of Pr\"ufer groups, we can strengthen the conclusion of Lemma \ref{lincom}.
\begin{lem}\label{lincom_gen}
Let $(G,<,+,0,\dots)$ be an asymmetric minimal group. Let $g$ be generic, $g_1,\ldots,g_k<g$, and $n_1,\ldots,n_k,n$ integers with $ng \ne 0$. Then
$n_1g_1+\cdots+n_kg_k\not= ng$, or $G$ is a finite sum of Pr\"ufer $p$-groups for some primes $p$ dividing $n$ (possibly with repetitions).\end{lem}
\begin{proof}
The proof is very similar to the previous one. Assuming for a contradiction that  $n_1g_1+\cdots+n_kg_k= ng$, we get that the set $X\subset G$ of realizations of the negation of the formula
\begin{equation}\label{eqn2}\exists x_1,\ldots, x_k\ \left(\bigwedge_{i=1}^kx_i<x\land\sum_{i=1}^kn_ix_i= nx\right)\end{equation}
is finite. Let $H$ be the $n$-divisible hull of the subgroup generated by $X$, i.e.,\ the collection of all elements $h \in G$ such that $n^mh \in \langle X \rangle$ for some $m\in \mathbb{N}$. If $H \ne G$, we finish as before by considering a minimal element $a \in G \setminus H$. So it remains to show that $H$ is a proper subgroup of $G$. By Reineke's result, either $G$ is elementary abelian of prime exponent $p$ and $H$ is finite (note that $ng \ne 0$ implies that $p$ does not divide $n$),
or $G$ is divisible with only finitely many elements of any given finite order. If it contains a copy $Q$ of $\mathbb Q$, then $(n+1)^{-k}1_Q\notin H$ for sufficienly big $k$. Otherwise $G$ contains a copy $P$ of a Pr\"ufer $p$-group for some $p$ not dividing $n$, so $H\cap P$ must be finite.\end{proof}

\begin{cor}\label{newlemma}
Let $(G,<,+,0,\dots)$ be an asymmetric minimal group. If $g$ is
generic and $h<g$, then $\pm g\pm h\sim g$.
\end{cor}
\begin{proof} Fix any choice of $\pm$ in $\pm g \pm h$. If $\pm g\pm h>g$, we have that $g,h<\pm g\pm h$ and $\pm g\pm h$ is generic; if $\pm g\pm h<g$, we have $h,\pm g\pm h<g$; in either case, we contradict Lemma \ref{lincom}.
\end{proof}

\begin{cor}\label{newcor}
Let  $(G,<,+,0,\dots)$ be an asymmetric minimal group and let $g$
be generic.  Then $g\nless ng$ for any integer $n$. If $ng\not=0$ and $G$ is not a sum of finitely many Pr\"ufer $p$-groups for primes $p$ dividing $n$, then $ng\nless g$, whence $g \sim ng$. Moreover, $g \sim g'$ for any $g' \in \frac{1}{n} g$.
\end{cor}

\begin{proof}
If $ng>g$, then clearly $ng$ is generic and it is a sum of strictly smaller elements, which contradicts Lemma \ref{lincom}.
If $ng<g$, we take $g_1=ng$ and $n_1=1$, contradicting Lemma \ref{lincom_gen}. Thus $g \sim ng$. If $g' \in \frac{1}{n}g$, then $g'$ is generic, so $g' \sim ng'=g$ by the first part of the proof.
\end{proof}

\begin{lem}\label{n-1sim} Let $G$ be an asymmetric minimal group, not of exponent dividing $n$, and $g$ generic. Then all elements of $\frac1n g$ are $\sim$-related.\end{lem}
\begin{proof} The finite set $X:=\{g':ng'=g\}$ must have a minimal element $g_0$ and a maximal element $g_1$, both generic over $G$. But
$$X=\{g':ng'=ng_0\}=\{g':ng'=ng_1\}$$
is invariant under an automorphism taking $g_0$ to $g_1$. Hence $g_0\sim g_1$.
\end{proof}

\begin{thm_proof}\label{torsion}
An almost linear minimal group $G$ is either elementary abelian, or it is a sum of finitely many Pr\"ufer $p$-groups for a fixed prime $p$. In particular, $G$ is a torsion group.
\end{thm_proof}

%%Frank: Slightly modified: Replcaed Q by P, since P is not used
\begin{proof}
Suppose $G$ is a counter-example. Then $G$ divisible, and at least one of the following cases holds:
\begin{enumerate}
\item[1.] $G$ contains a copy $P$ of a  Pr\"ufer $p$-group, but is not a sum of Pr\"ufer $p$-groups,
\item[2.] $G$ contains a copy $P$ of $\mathbb{Q}$; in this case, put $p=2$.
\end{enumerate}

In any of these cases, by Corollary \ref{newcor}, the set
$$X:=\{ x \in G : x \nsim y \;\; \mbox{for some}\;\; y \in {\textstyle\frac1p} x\}$$
is finite; let $H$ be the subgroup generated by $X$. Then $H\cap P$ is a proper subgroup of $P$, so there is $a_0 \in P\setminus H$. Choose $a_k\in P$ with $pa_{k+1}=a_k$ for all $k < \omega$. Then all $a_k$'s are outside $H$, so transitivity of $\sim$ implies that $\{a_k:k<\omega\}$ is a $\sim$-antichain. Since it must be finite,
there is $j>i$ with $a_i=a_j$. Then $p^{j-i}a_0=a_0$, a contradiction with the fact that the order of $a_0$ is a power of $p$ in Case 1 and is infinite in Case 2.
\end{proof}

This implies immediately the non-existence of almost linear minimal fields.
\begin{thm_proof}There is no almost linear minimal field.\end{thm_proof}
\begin{proof} The multiplicative group of an infinite field is neither elementary abelian, nor a finite sum of Pr\"ufer $p$-groups for a fixed prime $p$.\end{proof}

\section{Almost linear minimal groups as valued groups}\label{Section 5}

Recall that a {\em valued abelian group} is an abelian group $G$ together with a surjective valuation $v:G\to\Gamma$, where $\Gamma$ is a linearly ordered set with maximum $\infty$, such that:\begin{enumerate}
\item $v(x)=\infty$ if and only if $x=0$.
\item $v(x-y)\ge\min\{v(x),v(y)\}$.
\end{enumerate}
Note that the axioms imply $v(-x)=v(x)$ and $v(x-y)=\min\{v(x),v(y)\}$ unless $v(x)=v(y)$.
It follows that for every $\gamma\in\Gamma\setminus\{\infty\}$ the sets $\bar B(\gamma)=\{x\in G:v(x)\ge\gamma\}$ and $B^o(\gamma)=\{x\in G:v(x)>\gamma\}$ are subgroups of $G$. Valued abelian groups have been studied by Simonetta \cite{sim} and de Aldama \cite{deald}, who consider the following conditions for all primes $p$:
\begin{enumerate}
\item[(3)] $\forall x,y\ [v(px)<v(py)\to v(x)<v(y)]$.
\item[(4)] $\forall x,y\ [v(x)<v(y)\to(v(px)<v(py)\lor px=0)]$.
\item[(5)] $\forall x,y\ [v(x)<v(py)\lor \exists z\,pz=x]$.
\end{enumerate}
As $v$ is surjective, for all $n\in \mathbb N^*$ one can define a function $f_n:\Gamma\to\Gamma$ as $f_n((v(x))=v(nx)$.  In addition, for every $m\in\mathbb N^*$ we consider the unary relation $R_m$ on $\Gamma$ given by
$$R_m(x)\Leftrightarrow |\bar B(x)/B^o(x)|>m.$$
It follows from the axioms (1)--(5) that $f_n$ is well-defined and increasing, strictly so on $\Gamma\setminus f_n^{-1}(\infty)$. Moreover, if $f_n(\gamma)\not=\infty$, then $R_m(\gamma)\Leftrightarrow R_m(f_n(\gamma))$.
We put
$$\mathcal L_{vg}=\{+,0,v,\le,\infty\}\quad\text{and}\quad\mathcal L_v=\{\le,R_n,f_n:n\in\mathbb N^*\}.$$

Simonetta shows that if $G$ is a valued abelian group satisfying (1)--(5), then there is at most one prime $p$ such that $G$ is not $p$-divisible, and at most one prime $q$ such that $G$ has $q$-torsion. Moreover, he obtains the following relative quantifier elimination result:
\begin{fact}\cite[Theorem 3.3]{sim}\label{QE} Every $\mathcal L_{vg}$-formula $\phi(\bar x,\bar y)$ with variables $\bar x$ in the group sort and variables $\bar y$ in the value sort is equivalent in $G$ to some formula $\phi_v(v(t_1(\bar x)),\ldots,v(t_n(\bar x)),\bar y)$, where the $t_i(\bar x)$ are group terms in $\bar x$ and $\phi_v$ is an $\mathcal L_v$-formula. Moreover, $\phi_v$ and $t_1,\ldots,t_n$ only depend on $p$ and $q$.\end{fact}

Clearly a valued group with infinite value set is almost linear, where we take the inverse order induced from the valuation. For minimal groups, we shall now prove the converse: An almost linear minimal group $G$ carries interdefinably the structure of a valued group.
Recall that by Remark \ref{Remadditional} we can and do choose a definable order on $G$  so that $G/\!\sim$ is ordered in type $\omega$.

\begin{lem}\label{subgroup} Let $G$ be an almost linear minimal group and put $H_g=\{x\in G:x\not>g\}$. Then $H_g$ is a subgroup for almost all $g\in G$. The collection of these subgroups is linearly ordered by inclusion in order type $\omega$.\end{lem}
\begin{proof} Let $g$ be generic and $x,y\not> g$. Suppose for a contradiction that $x\pm y>g$. By transitivity of $\sim$, we have $x,y<x\pm y$, contradicting Lemma \ref{lincom}. Therefore $H_g$ is a subgroup for generic $g$, and thus for almost all $g$ by minimality. Clearly $g<g'$ implies $H_g \subset H_{g'}$, and $g\sim g'$ implies $H_g=H_{g'}$ by transitivity of $\sim$. Thus the set $\{H_g:g\in G\}$ has the same order type with respect to inclusion as $\{g/{\sim}:g\in G\}$ with respect to $<$, namely $\omega$.\end{proof}

\begin{prop}\label{valued} Let $G$ be an almost linear minimal group. After modifying the order on a finite set, the map $v:G\mapsto G/{\sim}$ endows $G$ with the structure of a valued abelian group, where $G/{\sim}$ is ordered by $g/{\sim}>g'/{\sim}$ if $g<g'$, of order type $\omega^*$. In particular, $G$ as almost linear group and $G$ as valued group are interdefinable (with parameters).\end{prop}
\begin{proof} Let $X\subseteq G$ be the co-finite set of $g\in G$ such that $H_g$ is a subgroup. We modify the order on $G$ so that $0$ is the unique minimal element. This can only increase $X$; in particular, $0\in X$ and $H_0=\{0\}$. Let $g_0$ be a minimal element greater than $G\setminus X$. We modify the order by making all elements of the finite set $H_{g_0}\setminus\{0\}$ incomparable. These modifications are clearly definable, the modified order on $G$ is still almost linear, and $H_g$ is now a subgroup of $G$ for all $g\in G$. Of course, $g<g'$ still implies $H_g \subset H_{g'}$, and $g\sim g'$ yields $H_g=H_{g'}$.

Now, $H_0=\{0\}$ implies axiom (1); the fact that $H_g$ is a subgroup for all $g\in G$ yields axiom (2). As we have only modified an initial segment of the order on $G$, the set of equivalence classes $G/{\sim}$ still has order type $\omega^*$. We have obtained the valuation $v$ definably from the almost linear structure; inversely, as we only modified the order on finitely many elements, we can define the original almost linear structure from the valuation.\end{proof}

%%Frank: next paragraph shortened
We say that Axioms (3)-(5) hold {\em generically} if they hold outside a finite set. It is clear that if Axiom (3) holds outside a finite set, then each function $f_n$ is well-defined outside a finite set $D_n$; we extend $f_n$ to the whole of $\Gamma$ by putting $f_n(D_n)=\infty$.

Let $G$ be a minimal valued group. Since $B^o(v(g))$ is a proper definable subgroup of $G$ and thus finite for all $g\in G$, it follows that $\Gamma$ has order type $\omega^*$ or is finite. In the latter case, if $\gamma\in\Gamma$ is the minimal element, then $\bar B(\gamma)=G$, so the valuation is determined by the restriction of $v$ to the finite group $B^o(\gamma)$ and hence definable in the pure group structure. Henceforth, we shall assume that $\Gamma$ has order type $\omega^*$, so $G$ is almost linear; by Theorem \ref{torsion} it is either elementary abelian of exponent $p$ or a finite product of Pr\"ufer $p$-groups, for some prime $p$. Note that for $n$ coprime to $p$ the function $f_n$ is just the identity as all groups $\bar B$ and $B^o$ are finite $p$-groups; this also shows that $v$ has finite fibres. Clearly, if $G$ has exponent $p$, then $f_p$ maps $\Gamma$ to $\infty$. We write $g\sim g'$ if $v(g)=v(g')$.

\begin{lem}\label{bounded} If $G$ is divisible and $g$ generic, then $\bar B(v(g))/\bar B(v(pg))$ is finite and isomorphic to $\bar B(v(pg))/\bar B(v(p^2g))$ as valued groups (with the induced valuation), via the map induced by $x\mapsto px$. In particular, the interval $[v(g),v(pg)]$ is finite and $\mathcal L_v$-isomorphic to $[v(pg),v(p^2g)]$. Hence for all but finitely many $\gamma$ the function $f_p$ is well-defined and corresponds to a right shift by $\ell$, where $[v(g),v(pg)]$ has length $\ell+1$. It is thus definable from the order. \end{lem}
\begin{proof} Let us check first that if $g$ is generic and $g\sim g'$, then $pg\sim pg'$. So consider $X:=\{pg':g'\sim g\}$; it must have a minimal element $g_0$ and a maximal element $g_1$, both generic over $G$. By Lemma \ref{n-1sim}, all elements of $\frac1p g_0$ and $\frac1p g_1$ are $\sim$-related to $g$, so
$$X=\{pg':g'\sim{\textstyle\frac1p} g_0\}=\{pg':g'\sim{\textstyle\frac1p} g_1\}$$
is invariant under an automorphism taking $g_0$ to $g_1$. Hence $g_0\sim g_1$.

Next, let us check that $pg\sim pg'$ implies $g\sim g'$ for generic $g$.
So consider a minimal element $g_0$ and a maximal element $g_1$ of the set $X:=\{g':pg'\sim pg\}$. Since
$$X=\{g':pg'\sim pg_0\}=\{g':pg'\sim pg_1\}$$
and the generic type is unique, we have $g_0\sim g_1$.

It follows that $x\mapsto px$ maps $\bar B(v(g))\setminus B^o(v(g))$ onto $\bar B(v(pg))\setminus B^o(v(pg))$. As it is a group homomorphism, it also maps $\bar B(v(g))$ onto $\bar B(v(pg))$, and hence $B^o(v(g))$ onto $B^o(v(pg))$. Thus $g'<g$ if and only if $pg'<pg$ for all $g'$. We conclude that the map induced by $x\mapsto px$ is an isomorphism  from $\bar B(v(g))/\bar B(v(pg))$ to $\bar B(v(pg))/\bar B(v(p^2g))$ as valued groups, and that $f_p$ yields an $\mathcal L_v$-isomorphism between $[v(g),v(pg)]$ and $[v(pg),v(p^2g)]$. By minimality, all of this holds for all $g$ outside some finite set $Y$.

It remains to show that $\bar B(v(g))/\bar B(v(pg))$ is finite.
Let $g_0\in G$ be such that $G[p]\cup Y\subseteq\bar B(v(g_0))$ (where $G[p]:=\{ x \in G: px=0\}$). Then for all $h>g_0$ the map $x\mapsto px$ maps the finite group $\bar B(v(h))$ onto the subgroup $\bar B(v(ph))$ and has finite kernel $G[p]$ independent of $h$. So $|\bar B(v(h))/\bar B(v(ph))|=|G[p]|$ for all $h>g_0$.\end{proof}

%%Frank: Modified the next sentence
An inspection of Simonetta's proof of Fact \ref{QE} shows that it yields the following proposition, as finite sets of exceptions can be dealt with definably.
\begin{prop}\label{QE2} Let $G$ be a valued abelian group whose valuation $v$ has finite fibres and whose value set $\Gamma$ has order type $\omega^*$.
Assume that $G$ is either elementary abelian of exponent $p$ or a finite product of Pr\"ufer $p$-groups for some prime $p$, and Axioms (3) and (4) hold generically, i.e.\ outside a finite set. Then every $\mathcal L_{vg}$-formula $\phi(\bar x,\bar y)$ with variables $\bar x$ in the group sort and variables $\bar y$ in the value sort is equivalent to a formula $\phi_v(v(t_1(\bar x)),\ldots,v(t_n(\bar x)),\bar y)$, where the $t_i(\bar x)$ are terms in $\bar x$ and $\phi_v$ is an $\mathcal L_v$-formula (both with parameters).\end{prop}

Notice that by Lemma \ref{bounded}, the assumptions of Proposition \ref{QE2} are satisfied for any valued minimal group with infinite value set $\Gamma$.

\begin{thm} A valued abelian group $G$ with infinite value set $\Gamma$ is minimal if and only if the induced $\mathcal L_v$-theory on $\Gamma$ is minimal of order type $\omega^*$, the map $v$ has finite fibres, and either $G$ is elementary abelian, or a finite product of Pr\"ufer $p$-groups for some prime $p$ and $f_p$ is eventually a well-defined $\mathcal L_v$-isomorphism acting by right shift.\end{thm}
\begin{proof} If $G$ is minimal as a valued group, then $\Gamma$ with the induced structure must be minimal, since an infinite co-infinite subset $X$ of $\Gamma$ has an infinite co-infinite pre-image $v^{-1}(X)$. We have seen above that $\Gamma$ has order type $\omega^*$, the group is either elementary abelian or a finite product of Pr\"ufer $p$-groups, and $v$ has finite fibres. The fact that  $f_p$ is eventually an $\mathcal L_v$-isomorphism acting by right shift follows from Lemma \ref{bounded}.

Conversely, suppose that the $\mathcal L_v$-structure $\Gamma$ is minimal of order type $\omega^*$, all fibres of $v$ are finite, and $G$ is elementary abelian of exponent $p$, or a finite product of Pr\"ufer $p$-groups and $f_p$ is eventually an $\mathcal L_v$-isomorphism acting by right shift. This implies that $f_n$ is the identity for $n$ coprime to $p$, and $f_p(\Gamma)=\infty$ if $G$ has exponent $p$. All of this implies that Axioms (3)-(5) hold generically.

Consider a formula $\phi(x,\bar g)$, where $\bar g$ are parameters in $G$. (Clearly, we can replace any parameter $\gamma\in\Gamma$ by some element of $v^{-1}(\gamma)$.)
By Proposition \ref{QE2} (enlarging the tuple $\bar g$ of parameters if necessary), this formula is equivalent to a formula $\phi_v(v(t_1(x,\bar g)),\ldots,v(t_k(x,\bar g)))$, where $\phi_v$ is an $\mathcal L_v$-formula and $t_1,\ldots,t_k$ are group terms.
As group terms are just $\mathbb Z$-linear combinations, there are integers $n_i\in\mathbb Z$ and $h_i\in\langle g_1,\ldots,g_k\rangle< G$ such that $t_i(x,\bar g)=n_ix+h_i$; if $G$ has exponent $p$, we may choose $0\le n_i<p$. Since $v$ has finite fibres, $\Gamma$ has order type $\omega^*$ and $G$ is elementary abelian or has finite $n$-torsion for all $n$, the set
$$X=\{g\in G: v(n_ig)\ge v(h_i)\text{ for some $i$ with }n_i\not=0\}$$
is finite. Let $Y$ be a finite subset of $G$ such that all the $f_{n_i}$ for $n_i>0$ are well-defined outside $v(Y)$. On $G\setminus (X \cup Y)$ the formula $\phi_v(v(t_1(x,\bar g)),\ldots,v(t_k(x,\bar g)))$ is equivalent to
$$\phi_v(f_{n_1}(v(x)),\ldots,f_{n_k}(v(x)))=\phi'_v(v(x),\bar\gamma),$$
where we have put $f_{n_i}(v(x))=v(h_i)=\gamma_i\in\Gamma$ whenever $n_i=0$. Since $\phi'_v(y,\bar\gamma)$ defines a finite or co-finite subset of $\Gamma$ and the fibres of $v$ are finite, $\phi'_v(v(x),\bar\gamma)$ defines a finite or co-finite subset of $G$. It follows that $\phi(x,\bar g)$ defines a finite or co-finite subset of $G$. Thus $G$ is minimal as a valued group.\end{proof}

This yields a classification of valued minimal groups:
\begin{thm} A valued group $G$ with infinite value set $\Gamma$ is minimal if and only if\begin{enumerate}
\item $\Gamma$ has order type $\omega^*$,
\item $v$ has finite fibres,
\item either $G$ is elementary abelian of exponent $p$ and
\begin{enumerate}
\item either there is $n_0<\omega$ such that $R_{p^{n_0}}\land\neg R_{p^{n_0+1}}$ is co-finite,
\item or $R_{p^n}\land\neg R_{p^{n+1}}$ is finite for all $n<\omega$,
\end{enumerate}
\item or $G$ is a finite product of Pr\"ufer $p$-groups, $f_p$ is eventually an $\mathcal L_v$-isomor\-phism acting by right shift, and there is $n_0<\omega$ such that $R_{p^{n_0}}\land\neg R_{p^{n_0+1}}$ almost everywhere.\end{enumerate}\end{thm}
\begin{proof} Clearly the conditions are necessary for $\Gamma$ to be minimal as an $\mathcal L_v$-structure; it is easy to see that they are also sufficient. Note that the analogue of option (3)(b) cannot occur in case (4), as $|\bar B(v(g))/\bar B(v(pg))|$ remains bounded by Lemma \ref{bounded}.\end{proof}

Note that all cases can be easily realised as a valued abelian group, which hence must be almost linear minimal. In particular, we obtain examples of almost linear minimal groups which are elementary abelian of exponent $p$ (for any prime $p$) as well as examples of almost linear minimal torsion groups of infinite exponent, namely finite products of Pr\"ufer $p$-groups (for any fixed prime $p$). This shows that the conclusion of Theorem \ref{T3} is strongest possible.

\end{document}